\documentclass[a4paper]{siamltex}

\usepackage{amsmath,amssymb,amsfonts}
\usepackage{comment,url}
\usepackage{pgfplots}
\usepackage{pgfplotstable}
\usepackage{booktabs}
\usepackage{tabularx}

\usepackage{graphicx,epstopdf}
\DeclareGraphicsExtensions{.pdf,.png,.jpg,.eps}
\usepackage{pgfplots}
\usepackage{pgfplotstable}
\usepackage{booktabs}
\usepackage{tabularx}
\pgfplotstableset{
  every head row/.style={before row=\toprule,after row=\midrule},
  clear infinite
}

\newcommand{\w}{_{_ \mathcal W}}
\newcommand{\ff}{_{_\mathcal F}}
\newcommand{\aqt}{_{_\mathcal{CQT}}}
\newcommand{\qt}{_{_\mathcal{QT}}}
\newcommand{\ZZ}{\mathbb Z}
\newcommand{\cu}{{\bf i}}

\newcommand{\norm}[1]{\lVert #1 \rVert}
\newcommand{\brb}{\ \rlap{\rule[0pt]{0.5pt}{0.9ex}}\rule[1.1ex]{0.5pt}{0.9ex}\ }

\newtheorem{remark}{Remark}

\begin{document}
\title{On Functions of quasi Toeplitz matrices\thanks{Work supported by GNCS of INdAM.}}
\author{Dario A. Bini
\thanks{Dipartimento di Matematica, Universit\`a di Pisa, ({\tt dario.bini@unipi.it})}
\and 
Stefano Massei
\thanks{Scuola Normale Superiore, Pisa ({\tt stefano.massei@sns.it})}
\and
Beatrice Meini
\thanks{Dipartimento di Matematica, Universit\`a di Pisa, ({\tt beatrice.meini@unipi.it})}}

\maketitle

\begin{abstract}Let $a(z)=\sum_{i\in\mathbb Z}a_iz^i$ be a complex valued continuous function, defined for $|z|=1$, such that $\sum_{i=-\infty}^{+\infty}|ia_i|<\infty$. Consider the semi-infinite Toeplitz matrix $T(a)=(t_{i,j})_{i,j\in\mathbb Z^+}$ associated with the symbol $a(z)$ such that $t_{i,j}=a_{j-i}$. A quasi-Toeplitz matrix associated with the continuous symbol $a(z)$ is a matrix of the form $A=T(a)+E$ where $E=(e_{i,j})$, $\sum_{i,j\in\mathbb Z^+}|e_{i,j}|<\infty$, and is called a CQT-matrix.
  Given a function $f(x)$ and a CQT matrix
  $M$, we provide conditions under which $f(M)$ is well defined and is
  a CQT matrix. Moreover, we introduce a parametrization of CQT
  matrices and algorithms for the computation of $f(M)$.  We treat the
  case where $f(x)$ is assigned in terms of power series and the case
  where $f(x)$ is defined in terms of a Cauchy integral.  This
  analysis is applied also to finite matrices which can be
  written as the sum of a Toeplitz matrix and of a low rank
  correction.
  
\end{abstract}


\section{Introduction}
Functions of finite matrices have received a lot of interest in the
literature both for their theoretical properties and for the many
applications they have in real world problems. We refer the interested reader to
the book \cite{higham2008functions} for more details.  
Among the available different definitions of
matrix function, which are equivalent under mild hypotheses, many rely
on the Jordan canonical form of the matrix argument or on its Schur
form. As a consequence, they are not directly generalizable to the
setting of infinite matrices. However, there are two definitions which
apparently seem to be more suited for extending the concept of
function to infinite matrices. They rely on the Laurent series
expansion, and on the integral
representation through the Dunford Cauchy formula
\cite{higham2008functions}.

In the set of matrices with infinite size there is a class which is a
cornerstone for numerical linear algebra.  It is the class of Toeplitz
matrices $T(a)=(t_{i,j})$ associated with a function
$a(z)=\sum_{i\in\mathbb Z}a_iz^i$, called symbol, defined by $t_{i,j}=a_{j-i}$.
Toeplitz matrices are widely analyzed in the literature, from the
pioneering papers by O. Toeplitz, to the seminal monograph by
Grenander and Szeg\H{o} \cite{grenander2001toeplitz}, until to the
more recent and wide contributions given by several international
research groups, including, but not limited to, the books \cite{boett},
\cite{bottcher2005spectral}, \cite{bottcher2012introduction},
\cite{chan2007introduction}, \cite{kailath} and the papers \cite{serra},
\cite{serra-tyrty}, \cite{tyrty}, \cite{widom} with the references cited therein. 
For a general view on
this topic we refer the reader to the monograph \cite{boett60} which
collects a wide set of contributions and gives the taste of the
problems in this area and an idea of their richness.

In some problems, typically encountered in the analysis of stochastic
processes, like bi-dimensional random walks in the quarter-plane, one has
to deal with matrices of the kind $A=T(a)+E$, where $T(a)$ is a
semi-infinite Toeplitz matrix, while $E=(e_{i,j})$ is a non-Toeplitz
correction such that $\sum_{i,j\in\mathbb Z^+}|e_{i,j}|<\infty$. We
call this class of matrices {\em Quasi-Toeplitz} matrices, in short QT
matrices, and denote them with $\mathcal{QT}$. We consider the subset of QT matrices whose symbol $a(z)$ is differentiable and both $a(z)$ and $a(z)'$ belong to the Wiener class. We call the latter as the set of CQT matrices and we denote it with $\mathcal{CQT}$. 

From the computational point of view, one has to solve matrix
equations where the coefficients are CQT matrices, or to compute the
value of $f(A)$, where $f(z)$ is an assigned function, typically the
exponential function, and $A$ is a CQT matrix.

Concerning the recent literature in this research area, it is worth
citing the paper \cite{bini2016computing} where the exponential
function of a block-triangular block-Toeplitz matrix is analyzed with
application to solving certain fluid queues. In the recent paper
\cite{kressner2016fast} the problem of computing the exponential
function of finite Toeplitz matrices is investigated and several
applications are presented. In \cite{bm:exp} the case of the
exponential of a semi-infinite CQT matrix is analyzed in depth.
Related issues concern the decay of the coefficients of a matrix function \cite{BeBo14},
\cite{BeSi15}, \cite{PoSi}, and the decay of the singular values of matrices having a displacement rank structure \cite{becker}. 

In this paper, our interest is addressed to analyze the definition and 
the properties of $f(A)$, for $A$ being a CQT matrix.
This work continues the analysis recently started in \cite{bmm} and in 
\cite{bm:exp} where the classes of QT and CQT matrices have been introduced and analyzed, and 
where the exponential function has been extended to the case of CQT matrices.

Here we deepen this analysis by considering matrix functions assigned either in terms of a
Laurent power series or in terms of the Dunford-Cauchy integral. We give conditions on the function
$f(z)$ and on the CQT matrix $A$ in order to ensure that the semi-infinite
matrix $f(A)$ is well defined and is still a CQT matrix, i.e., it can be written in the form $f(A)=T(f(a)) +E_{f(a)}$. Moreover, we outline some
algorithms for its computation.  Finally, we show that the same analysis
can be applied to the case where the matrix is finitely large and can still be written as the sum of a Toeplitz matrix and a correction. 

A case of interest concerns matrices associated with an analytic
symbol $a(z)$ where the coefficients of the Toeplitz part have an
exponential decay. This situation is very convenient from the
computational point of view. However, the class of matrices that we
obtain this way, which we call analytically quasi-Toeplitz (AQT), is
still a matrix algebra, but is not a Banach space with the norm
$\|\cdot\|\aqt$. In the analysis that we carry out, we point out the
cases where the result of the computation is still in the class of AQT
matrices.

The paper is organized as follows. In Section \ref{sec:2} we recall
the representation of semi-infinite CQT matrices and the description of the arithmetic in this matrix algebra. Moreover, we outline the way to extend this definition to the case 
of finitely large matrices where $A$ can still be written as $A=T(a)+E$, and the correction $E$ has nonzero entries in the upper leftmost corner and in the lower rightmost corner, so that the rank of $E$ is small with respect to its size.

 In Section \ref{sec:3} we generalize the concept of matrix
functions to semi-infinite CQT matrices. As first step, we consider the case where the function $f(x)$ is assigned in terms of a power series and give conditions under which $f(A)$ is still a CQT matrix. Then we analyze the case where $f(x)$ is given in terms of a Laurent series. Also in this framework, we give conditions under which $f(A)$ is a CQT matrix.
We examine some computational issues and outline an algorithm for computing a general matrix function of a CQT matrix. Some numerical tests which demonstrate the effectiveness of our approach are presented both for semi-infinite and for finite matrices.

 Section \ref{sec:3.2} is devoted to the extension of the
concept of matrix function based on the Dunford-Cauchy integral. Once
again, we provide conditions under which $f(A)$ is a CQT
matrix. Finally,  we discuss some computational
issues, outline an algorithm for computing $f(A)$ based on the trapezoidal rule, and  present a validation of our theory by means of some
numerical tests. In Section \ref{sec:conc} we draw some conclusions and research lines.

\section{Definitions and preliminaries}\label{sec:2}
We indicate with $\mathbb T=\{z\in\mathbb C:\quad |z|=1\}$ the unit
circle of the complex plane $\mathbb C$, with $\mathbb Z$ and $\mathbb
Z^+$ the ring of relative integers and the set of positive integers,
respectively. We denote by $\mathcal W$ the Wiener class, i.e., the set of
complex valued Laurent series $a(z)=\sum_{i\in\mathbb Z}a_iz^i$ such that $a_i\in\mathbb C$ and
$\|a\|\w:=\sum_{i\in\mathbb Z}|a_i| <+\infty$. We also define the set $\mathcal W_1:=\{a(z)\in\mathcal W:\quad a'(z)\in\mathcal W\}\subset\mathcal W$. Recall that
$\mathcal W$ is a Banach algebra with the norm $\|a\|\w$, i.e., a
Banach linear space closed under product of functions.  Given
$a(z)\in\mathcal W$ define $a^+(z)=\sum_{i\in\mathbb Z^+}a_iz^i$,
$a^-(z)=\sum_{i\in\mathbb Z^+}a_{-i}z^i$ so that $a(z)=a_0+a^-(z)+a^+(z)$, and associate with $a(z)$
and $a^\pm(z)$ the following semi-infinite matrices
\[
\begin{array}{ll}
T(a)=(t_{i,j}), & t_{i,j}=a_{j-i},\\[1ex]
H(a^+)=(h^+_{i,j}),& h^+_{i,j}=a_{i+j-1},\\[1ex]
H(a^-)=(h^-_{i,j}),& h^-_{i,j}=a_{-i-j+1}.
\end{array}
\]
Observe that $T(a)$ is Toeplitz while $H(a^+)$ and $H(a^-)$ are Hankel
matrices.  The function $a(z)$ is said {\em symbol} associated with $T(a)$.
Moreover, denote by $\mathcal F$ the set of semi-infinite
matrices $F=(f_{i,j})_{i,j\in\mathbb Z^+}$ such that $\|F\|\ff
:=\sum_{i,j\in\mathbb Z^+}|f_{i,j}|<+\infty$.

\subsection{Quasi Toeplitz matrices}
Dealing with an infinite amount of data can be an insurmountable
problem from the computational point of view. However, such
difficulties can be overcome whenever the data to be processed are
representable ---at an arbitrary precision--- with a finite number of
parameters. With this motivation, in \cite{bmm} the authors introduce
the following classes of infinite matrices.

\begin{definition}\label{def:qt}
A semi-infinite matrix $A$  is said {\em quasi Toeplitz (QT-matrix)} if it can be written in the form
\[
A=T(a)+E,
\]
where $a(z)=\sum_{i\in\mathbb Z}a_iz^i\in\mathcal W$
and $E=(e_{i,j})\in\mathcal F$.  We refer to $T(a)$ as the Toeplitz
part of $A$, and to $E$ as the correction. If $a(z)$ belongs to $\mathcal W_1$ then we say that $A$ is a CQT matrix.
Finally, if $a(z)$ is analytic in an open annulus enclosing $\mathbb T$,
then the matrix $A$ is said {\em analytically quasi Toeplitz (AQT-matrix)}.
The classes of QT-matrices, CQT-matrices and AQT-matrices are denoted by
$\mathcal{QT}$, $\mathcal{CQT}$ and $\mathcal{AQT}$, respectively.
\end{definition}

We report some properties of these classes. More details are given in \cite{bmm} and
\cite{bm:exp}. The first result shows that $\mathcal{QT}$ is a Banach space,
while $\mathcal{CQT}\subset\mathcal{QT}$ is a Banach algebra.

\begin{theorem}\label{thm:qt}
  The class $\mathcal{QT}$ is a Banach space with the norm
  $\|T(a)+E\|\qt:=\|a\|\w+\|E\|\ff$. The class $\mathcal{CQT}$ is a Banach algebra equipped with the
  norm $\|T(a)+E\|\aqt:=\|a\|\w+\|a'\|\w+\|E\|\ff$, where $a'(z)$ is the
  first derivative of $a(z)$. The class $\mathcal AQT$ is a normed algebra equipped with the norm $\|\cdot\|\aqt$. Moreover,
  $\|AB\|\aqt\le\|A\|\aqt\|B\|\aqt$ for any $A,B\in\mathcal{CQT}$, and $\|A\|\qt\le\|A\|\aqt$ for any $A\in\mathcal{CQT}$. 
\end{theorem}

\begin{remark}\label{rem:cauchy}\rm
It is interesting to notice that $\mathcal {AQT}$ with the norm $\|\cdot\|\aqt$ is not Banach. In fact, consider the sequence of semi-infinite Toeplitz matrices $\{T(a_n)\}$ with $a_n(z)=\sum_{j=1}^n\frac1{j^3} z^j$, 
and observe that this is a Cauchy sequence in $\mathcal{AQT}$ with the norm $\|\cdot\|\aqt$, but its limit does not belong to $\mathcal{AQT}$ because the corresponding symbol $a(z)=\sum_{j=1}^\infty \frac1{j^3}z^j$ is not analytic. 
On the other hand, the completeness of $\mathcal{CQT}$ implies that any Cauchy sequence in $\mathcal{AQT}$ admits limit in $\mathcal{CQT}$. Therefore, we can claim that the limit of a Cauchy sequence in $\mathcal{AQT}$ can still be represented ---at an arbitrary precision--- with a finite number of parameters. 
\end{remark}\ \\

The next result from \cite{bottcher2005spectral} provides a representation of the product of two semi-infinite Toeplitz matrices.

\begin{theorem}\label{th:1}
 For $a(z),b(z)\in\mathcal W$ let $c(z)=a(z)b(z)$. Then we have
 \[
 T(a)T(b)=T(c)-H(a^-)H(b^+).
 \]
 \end{theorem}

The following result from \cite{bm:exp} provides a representation of the matrices $T(a)^i$ and $(T(a)+E)^i$. 

\begin{theorem}\label{thm:aqt1}
  If $a(z)\in\mathcal W_1$  then
$
T(a)^i=T(a^i)+E_i,
$
where $E_1=0$ and
$
E_{i}=T(a)E_{i-1}-H(a^-)H((a^{i-1})^+),\quad i\ge 2
$.
Moreover,   
\[
\|E_i\|\ff\le \frac{i(i-1)}2 \|a'\|\w^2\|a\|^{i-2}\w.
\]
 If $A=T(a)+E\in\mathcal{CQT}$ then
$
A^i=T(a^i)+D_i,
$
where $D_0=E$ and 
\[
D_i=A D_{i-1}-H(a^-)H((a^{i-1})^+)+ET(a^{i-1}),\quad i\ge 1.
\]
Moreover, for $\alpha=\|a'\|^2\w+\|E\|\ff,\quad \beta=\|a'\|\w^2$ we have
\[
\|D_i\|\ff\le\frac{1}{\|E\|\ff}\left(\alpha\frac{(\|a\|\w+\|E\|\ff)^i-\|a\|^i\w}{\|E\|\ff}-\beta i\|a\|^{i-1}\w\right).
\]
\end{theorem}

\subsection{Quasi-Toeplitz matrix arithmetic}\label{sec:arith}
 Since $\mathcal{CQT}$ is a matrix algebra then the outcome of a rational computation which takes as input a CQT-matrix and that can be carried out with no breakdown, say caused by singularity, still belongs to $\mathcal{CQT}$. 
 This observation explains how to compute  functions of semi-infinite CQT-matrices by means of convergent sequences of rational ap\-prox\-i\-ma\-tions.

For this reason, in \cite{bmm} the design and the analysis of a matrix arithmetic
 for CQT matrices has been introduced. In particular, a CQT matrix $A=T(a)+E$  
 is represented in the following form. The Toeplitz part $T(a)$ is represented up
  to an arbitrarily small error by means of the finite sequence 
  $\widehat a=(a_{-n_-},\ldots,a_0,\ldots,a_{n_+})$ of the coefficients of the Laurent
   polynomial $\widehat a(z)=\sum_{i=-n_-}^{n_+}a_iz^i$ which approximates the
     function $a(z)=\sum_{i\in\mathbb Z}a_iz^i$ for suitable $n_-,n_+ \ge
     0$. The correction $E$ is represented by two finite dimensional matrices $F$ and $G$ such 
     that $FG^T$ coincides with the non negligible part of $E$. Here $F$ and $G$ 
     have size $n_f\times r$, $n_g\times r$, respectively where $n_f$ and $n_g$ are 
     the sizes of the leading submatrix of $E$ which contains the non-negligible 
     entries, while $r$ is the rank of this submatrix.

This way, a matrix in $\mathcal{CQT}$ is represented by means of the triple $(\widehat a,E,F)$ up to any controllable small error $\epsilon$.
Each arithmetic operation, say, $C=A+ B$, is defined by means of a set of equations which provide the components of the triple $(\widehat c,F_c,G_c)$ associated with  the matrix $C$ in terms of the components of the triples $(\widehat a,F_a,G_a)$ and $(\widehat b,F_b,G_b)$ associated with $A$ and $B$, respectively.
Similarly, matrix inversion is implemented by providing the relation between the triple $(\widehat a,F_a,G_a)$ associated with $A$ and the triple $(\widehat c,F_c,G_c)$ associated with $C=A^{-1}$.
An operation of compression is introduced to reduce the sizes of $F_c$ and $G_c$ obtained after each operation. More details can be found in \cite{bmm}.

 \subsubsection{Finite quasi Toeplitz matrix arithmetic}\label{sec:finite}
 Given a symbol $a(z)$ and $m\in\mathbb Z^+$ we indicate with
 $T_m(a)$, $H_m(a^-)$ and $H_m(a^+)$ the $m\times m$ leading principal
 submatrices of $T(a)$, $H(a^-)$ and $H(a^+)$, respectively.
 
 The approach used in Section~\ref{sec:arith} can be easily adapted to
 design a matrix arithmetic for quasi Toeplitz matrices of finite
 size.  The crucial tool for doing this extension is a version of
 Theorem \ref{th:1} which applies to the finite case.
 
\begin{theorem}\label{thm:toep-mul}
  Let $a(z)=\sum_{i=-m+1}^{m-1}a_iz^i$, $b(z)=\sum_{i=-m+1}^{m-1}b_iz^i$ and set  
$c(z)=a(z)b(z)=\sum_{i=-2m+2}^{2m-2}c_iz^i$. Then, 
  \[
  T_m(a)T_m(b)=T_m(c)-H_m(a^-)H_m(b^+)-JH_m(a^+)H_m(b^-)J,
  \]
   where $J$ is the $m\times m$ flip matrix having 1 on the anti-diagonal and zeros elsewhere.
  \end{theorem}

A pictorial description of the above theorem is given in Figure \ref{fig:0} where $a(z)$ and $b(z)$ are Laurent polynomials of degree much less than $m$ so that $T_m(a)$ and $T_m(b)$ are banded.

  \begin{figure}[!ht]
                \centering
                \includegraphics[width=0.4\textheight]{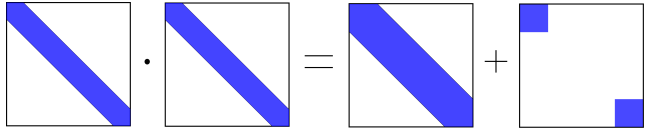}
                \caption{Multiplication of two finite dimensional banded Toeplitz matrices}\label{fig:0}
  \end{figure}
                
If $a(z)=\sum_{i=-k}^ka_iz^i$, $b(z)=\sum_{i=-k}^k b_iz^i$ with $k$
much smaller than $m-1$, then the matrices $H_m(a^-)H_m(b^+)$ and
$JH_m(a^+)H_m(b^-)J$ have disjoint supports located in the upper leftmost
corner and in the lower rightmost corner, respectively. Thus,
$T_m(a)T_m(b)$ can be represented as the sum of the Toeplitz matrix
associated with the Laurent polynomial $c(z)$ and of two correction
matrices $E^+$ and $E^-$ which collect the finite number of nonzero
entries located in the upper leftmost and in the lower rightmost
corners, respectively.
 
Algorithms for dealing with finite quasi Toeplitz matrices can be
easily obtained from those presented in \cite{bmm} just by taking into
account the additional lower rightmost corner correction. In the case
of a sufficiently large gap between the size $m$ and the bandwidth $k$
of the symbols that come into play, the two corner corrections behave
independently of each other and the finite CQT matrix arithmetic
becomes more effective. The cost of these operations essentially
depends on the bandwidth of $T_m(c)$ and on the sizes and ranks of the
correction matrices $E^+$ and $E^-$. The cost remains small as long as
the bandwidth and the size of the corrections $E^+$, $E^-$ remain
small together with their rank.  Whether this condition is not
satisfied, the two corrections may spread and overlap. This may cause
a slowdown due to the additional operations which are needed in the
computation.

\section{Function of a CQT matrix: power series representation}\label{sec:3}
In this section we give conditions under which a function $f(x)$, expressed in terms of a power series or a Laurent series, can be applied  to matrices $A$ in the class $\mathcal{CQT}$, and prove that under these conditions $f(A)$ still belongs to $\mathcal {CQT}$.

Let $a(z)\in\mathcal W_1$ and
 $A=T(a)+E\in\mathcal{CQT}$. Assume we are given a complex valued function
$f(x)=\sum_{i=0}^{+\infty} f_ix^i$ which is
 analytic on the open disk  $\mathbb D(\rho)=\{x\in\mathbb C:\quad |x|< \rho\}$.
 Observe that, if  $a(\mathbb T)\subseteq \mathbb D(\rho)$,
 then the composed function $f(a(z))$ belongs to $\mathcal W_1$. Recall also the following decay property of the coefficients $f_i$ of an analytic function $f(x)$ on $\mathbb D(\rho)$, see \cite[Theorem 4.4c]{henrici}:
 \begin{equation}\label{eq:decay}
 \forall \ \epsilon>0,\ \epsilon<\rho,\ \exists \gamma>0\ : \ |f_i|\le\gamma (\rho-\epsilon)^{-i}.
 \end{equation}

Define $\varphi_k(x)=\sum_{i=0}^kf_ix^i$ and observe that for any integers $h,k$ such that $h>k$ one has $\varphi_h(A)-\varphi_k(A)=\sum_{i=k+1}^h f_iA^i$. Thus, 
\begin{equation}\label{eq:phi}
\|\varphi_h(A)-\varphi_k(A)\|\aqt\le
\sum_{i=k+1}^h |f_i|\cdot\|A\|\aqt^i.
\end{equation}

This inequality implies the following result.

\begin{theorem}\label{th:7}
  Let $A=T(a)+E_a\in\mathcal{CQT}$ and let
  $f(x)=\sum_{i=0}^{+\infty}f_ix^i$ be analytic in $\mathbb
  D(\rho)$. If $\|A\|\aqt<\rho$ then $f(A)=\sum_{i=0}^{+\infty}f_iA^i$
  is well defined, belongs to $\mathcal  {CQT}$, and
  \[
  f(A)=T(f(a))+E_{f(a)},\quad E_{f(a)}\in\mathcal F.
  \] 
  
  Furthermore, if $A\in\mathcal{AQT}$ then $f(A)\in\mathcal{AQT}$. More precisely,
  there exists an annulus $\mathbb A(r,R)$ containing $\mathbb T$, such that $f(a(z))$ is well defined and analytic for
  $z\in\mathbb A(r,R)$.

\end{theorem}

\begin{proof}
We prove that the sequence $\varphi_k(A)=\sum_{i=0}^kf_iA^i$ is a Cauchy sequence in $(\mathcal{CQT},\|\cdot\|\aqt)$. In fact, since
$\|A\|\aqt<\rho$ there exists $0<\delta<\rho$ such that $\|A\|\aqt =\rho-\delta$. Thus, from \eqref{eq:phi}, for $h>k$ we have $\|\varphi_h(A)-\varphi_k(A)\|\aqt
\le \sum_{i=k+1}^{h}|f_i|(\rho-\delta)^i$. On the other hand, in view of equation \eqref{eq:decay} with $\epsilon=\delta/2$, there exists $\gamma$ such that $|f_i|\le\gamma (\rho-\delta/2)^{-i}$. This implies that 
$\|\varphi_h(A)-\varphi_k(A)\|\aqt
\le \gamma\sum_{i=k+1}^{h}\lambda^i$, $\lambda=(\rho-\delta)/(\rho-\delta/2)<1$. Thus for sufficiently large values of $h$ and $k$, the latter summation is smaller than any given $\epsilon>0$ so that the sequence
$\varphi_k(A)$ is Cauchy. Since the space $\mathcal{CQT}$ is Banach,
there exists $F\in\mathcal{CQT}$ such that
$\lim_k\|\varphi_k(A)-F\|\aqt=0$. That is, $F:=f(A)$ is
well defined and belongs to $\mathcal{CQT}$. Thus, $f(A)$ can be written as $f(A)=T(g)+E_g$ for a suitable $g(z)\in\mathcal W_1$ and $E_g\in\mathcal F$. 
Observe that $\varphi_k(A)$ can be written in the form
$\varphi_k(A)=T(\varphi_k(a))+E_k$ for a suitable $E_k\in\mathcal F$. Thus, the convergence of 
$\varphi_k(A)$ to 
$T(g)+E_g$ in the norm $\|\cdot\|\aqt$ implies that 
$\lim_k\|E_k-E_g\|\ff=0$ and $\lim_k\|\varphi_k(a)-g\|\w=0$. Thus we deduce that $g(z)=f(a(z))$.
In the case $A\in\mathcal{AQT}$, in order to show that 
$F\in\mathcal{AQT}$, it is sufficient to prove that $g(z)=f(a(z))$ is analytic over some annulus $\mathbb A(r,R)$.
From the condition $\|a\|\w\le \|A\|\aqt<\rho$ it follows that  for $|z|=1$, we have $|a(z)|\le\sum_{i\in\mathbb Z}|a_i|\cdot|z|^i=\|a\|\w<\rho$. By continuity of $a(z)$ there exists an open annulus $\mathbb A(r,R)$ which includes the unit circle $\mathbb T$, such that $|a(z)|<\rho$ for $z\in\mathbb A(r,R)$. This way, the function $f(a(z))$ is well defined and analytic in $\mathbb A(r,R)$ since composition of analytic functions. This shows that $f(A)\in\mathcal{AQT}$ and the proof is complete. 
\end{proof}

 Now we consider the problem of determining bounds to $\|E_{f(a)}\|\ff$. 
These bounds are useful from the computational point of view since they provide an indication of the mass of information which is stored in the correction part of $f(A)$.  Equivalently, they tell us how much the matrix $f(A)$ differs from a Toeplitz matrix. 
 For simplicity,  we deal with the case where $A=T(a)$ is Toeplitz. Then we treat the general case of a matrix $A=T(a)+E_a$.

Since $\|a\|\w< \rho$, for the analyticity of $f(x)$ in the disk $\mathbb D(\rho)$, we may write
\[
\|f(a)\|\w\le\sum_{i=0}^{+\infty} |f_i|\cdot \|a\|\w^i<\sum_{i=0}^{+\infty} |f_i|\rho^i <\infty.
\]

Let $A^k=T(a^k)+E_k$ and decompose $\varphi_k(A)$ as
$\varphi_k(A)=G_k+F_k$, where $G_k=\sum_{i=0}^k f_i T(a^i)$, $F_k=\sum_{i=0}^k f_iE_i$. Then we have
$G_k=T(\sum_{i=0}^kf_ia^i)$ so that, $\lim_{k}G_k=T(f(a))$ and
 $\lim_{k}F_k=f(A)-T(f(a))=E_{f(a)}$.

Now we can provide upper bounds for $\|E_{f(a)}\|\ff$ in the case of an almost general function $f(z)$. 

\begin{theorem}\label{thm:5.1}
Assume that the function  $f(x)=\sum_{i\in\ZZ^+}f_ix^i$ is analytic on $\mathbb D(\rho)$, that $a(z)\in\mathcal W_1$ and is such that 
$\|a\|\w<\rho$. Let $A=T(a)$ and $f(A)=T(f(a))+E_{f(a)}$. Then
\[
\|E_{f(a)}\|\le \frac12 \|a'\|\w^2 g''(\|a\|\w)
\]
where $g(z)=\sum_{i=0}^\infty |f_i| z^i$.
\end{theorem}

\begin{proof}
Recall from Theorem \ref{th:7} that $f(a(z))\in\mathcal W_1$.
From Theorem \ref{thm:aqt1} we have the bound 
\[
\|E_i\|\ff\le \frac{i(i-1)}2 \|a'\|\w^2\|a\|\w^{i-2}
\]
so that for the matrix $E_{f(a)}=\sum_{i=0}^\infty f_i E_i$ we have
\[
\|E_{f(a)}\|\ff\le \sum_{i=0}^\infty |f_i|\cdot \|E_i\|\ff\le 
\frac12 \|a'\|\w^2 \sum_{i=0}^\infty  i(i-1)|f_i|\cdot\|a\|\w^{i-2}
=\frac12 \|a'\|\w^2 g''(\|a\|\w),
\]
where $g''(\|a\|\w)$ is well defined and finite since 
$\|a\|\w<\rho$ and $f(z)$ is analytic for $|z|\le\rho$. This
 completes the proof.
\end{proof}

Observe that in the case of a power series with non-negative coefficients $f_i$ we have $g(x)=f(x)$. In particular, for $f(x)=e^x$ we get $\|E_{\exp(a)}\|\ff\le \frac12 \|a'\|\w^2
\exp(\|a\|\w)$, which coincides with the bound given in \cite{bm:exp}.

In the case where $A=T(a)+E_a$, we may prove a similar bound relying on
Theorem \ref{thm:aqt1} as expressed by the following

\begin{theorem}\label{thm:5.2}
Assume that the function  $f(x)=\sum_{i\in\ZZ^+}f_ix^i$ is analytic on $\mathbb D(\rho)$, that $a(z)\in\mathcal W_1$, and is such 
that
$\|a\|\w<\rho$. Let $A=T(a)+E_a$ and $f(A)=T(f(a))+E_{f(a)}$. Then
\[
\|E_{f(a)}\|\ff\le \frac{1}{\|E_a\|\ff}\left(\alpha\frac{ g(\|a\|\w+\|E_a\|\ff)-g(\|a\|\w)}{\|E_a\|\ff} -\beta g'(\|a\|\w)\right)
\]
where $g(z)=\sum_{i=0}^\infty |f_i| z^i$ and $\alpha=\|a'\|^2\w+\|E_a\|\ff,\quad \beta=\|a'\|\w^2$.
\end{theorem}

\begin{proof}
Recall from Theorem \ref{th:7} that $f(a(z))\in\mathcal W_1$  and that $E_{f(a)}=\lim _{k} F_k$, $F_k=\sum_{i=0}^k f_iD_i$ for $A^i=T(a^i)+D_i$.
From Theorem \ref{thm:aqt1} we have the bound 
\[
\|D_i\|\ff\le \frac{1}{\|E_a\|\ff}\left(\alpha\frac{(\|a\|\w+\|E_a\|\ff)^i-\|a\|^i\w}{\|E_a\|\ff}-\beta i\|a\|^{i-1}\w\right)
\]
with $\alpha=\|a'\|^2\w+\|E_a\|\ff,\quad \beta=\|a'\|\w^2$
so that for the matrix $E_{f(a)}=\sum_{i=0}^\infty f_i D_i$ we get the bound
$\|E_{f(a)}\|\ff\le\sum_{i=0}^\infty |f_i|\cdot\|D_i\|\ff$ which leads to
\[
\|E_{f(a)}\|\ff\le \frac{1}{\|E_a\|\ff}\left(\alpha\frac{ g(\|a\|\w+\|E_a\|\ff)-g(\|a\|\w)}{\|E_a\|\ff} -\beta g'(\|a\|\w)\right).
\]
 This
 completes the proof.
\end{proof}

Observe that, taking the limit for $\|E_a\|\ff\to 0$ in the bound given in the above theorem yields the bound of Theorem \ref{thm:5.1}.

Next, we consider the case where the function $f(x)$ is assigned as a Laurent series in the form $f(x)=\sum_{i\in\mathbb Z}f_ix^i$ analytic over the open annulus $\mathbb A(r_f,R_f)$ for $r_f<R_f$. 
We recall from \cite[Theorem 4.4c]{henrici} the following decay property of the coefficients $f_i$: 
\begin{equation}\label{eq:decay1}
 \forall \ \epsilon>0,\ \epsilon<R_f,\ \exists \gamma>0\ : \ |f_{i}|\le\gamma (R_f-\epsilon)^{-i},\quad
 |f_{-i}|\le\gamma (r_f+\epsilon)^i,\quad i>0.
 \end{equation}

Concerning the existence of $f(A)$ 
for $A\in\mathcal{AQT}$ we have the following

\begin{theorem}\label{thm:func}
Let  $f(x)=\sum_{i\in\mathbb Z}a_ix^i$ be an analytic function
in the open annulus $\mathbb A(r_f,R_f)$. Let $a(z)\in\mathcal W_1$  and consider a matrix $A=T(a)+E_a\in \mathcal{CQT}$. 
If  $a(\mathbb T)\subset \mathbb A(r_f,R_f)$,
 $\norm{A^{-1}}\aqt<r_f^{-1}$ and 
 $\norm{A}\aqt<R_f$ then
\[
f(A):=\sum_{i\in\mathbb Z}a_iA^i= 
T(f(a))+E_{f(a)}\in\mathcal{CQT}.
\]
Moreover if $A\in\mathcal{AQT}$ then $f(A)\in\mathcal{AQT}$.
\end{theorem}

\begin{proof}
The proof follows the same line as the one of Theorem \ref{th:7}. We consider $\varphi_k(x)=\sum_{i=-k}^k f_ix^i$ and show that $\varphi_k(A)$ is a Cauchy sequence in $\mathcal{CQT}$. Since $\|A^{-1}\|\aqt<r_f^{-1}$ and $\|A\|\aqt<R_f$,  there exists $0<\delta<R_f$ such that $\|A^{-1}\|\aqt\le (r_f+\delta)^{-1}$ and $\|A\|\aqt\le R_f-\delta$. Thus, applying the inequality \eqref{eq:decay1} with $\epsilon=\delta/2$, for $h>k>0$ we get 
\[\begin{split}
\|\varphi_k(A)-\varphi_h(A)\|\qt &\le\sum_{i=k-1}^h(|f_i|\cdot\|A\|\aqt ^i+|f_{-i}|\cdot\|A^{-1}\|^i\aqt)\\
&\le\gamma \sum_{i=k-1}^h\left(\left(\frac{R_f-\delta}{ R_f-\delta/2}\right)^{i}+\left(\frac
{r_f+\delta/2}{ r_f+\delta}\right)^{i}\right).
\end{split}
\] 
The latter quantity converges to 0 for $k\to\infty$ so that the sequence $\varphi_k(A)$ is Cauchy in $\mathcal CQT$ and thus there exists $F\in\mathcal CQT$ such that $\lim _{k\to\infty}\|\varphi_k(A)-F\|\qt =0$. Therefore the matrix $F$ has the form $F=T(g)+E_g$ for some function $g$ in the Wiener class and for $E_g\in\mathcal F$. By using the same argument as in the proof of Theorem \ref{th:7}, we obtain that $g(z)=f(a(z))$.

Now, consider the case $a(z)$ analytic. Since $a(\mathbb T)\subset \mathbb A(r_f,R_f)$, then there exists an open annulus $\mathbb A(r,R)$, which includes the unit circle, such that $a(\mathbb A(r,R))\subseteq \mathbb A(r_f,R_f)$ so that $f(a(z))$ is well defined and analytic for $z\in\mathbb A(r,R)$. Thus  $g(z)=f(a(z))$ is analytic for $z\in\mathbb A(r,R)$. Therefore we may conclude that $F\in\mathcal{AQT}$.
\end{proof}

Observe that
 the two technical hypotheses 
\[
\norm{A^{-1}}\aqt<r_f^{-1},\quad \norm{A}\aqt<R_f,
\]
given in Theorem~\ref{thm:func},
are not needed if $f(x)$ is a Laurent polynomial, i.e., a function of the form $f(x)=\sum_{i=-n_1}^{n_2}f_ix^i$.
If the function is entire on $\mathbb C$, we need no additional assumption. For example we can claim that the exponential function of a $\mathcal{CQT}$-matrix is 
again a $\mathcal{CQT}$-matrix.

\subsection{Computational aspects}\label{sec:exp1}
Observe that, if $f(x)=\sum_{i=0}^\infty f_ix^i$ and $A=T(a)$, the
combination of the two expressions $\varphi_k(A)=\sum_{i=0}^k f_iA^i$
and $A^i=T(a^i)+E_i$, enables one to compute the quantity
$\varphi_k(A)$ at a low computational effort. In fact, decomposing
$\varphi_k(A)$ as $\varphi_k(A)=T(\varphi_k(a))+F_k$, from
$\varphi_{k+1}(A)=\varphi_k(A)+f_{k+1}A^{k+1}$ we deduce the equation
\[
F_{k+1}=F_k+f_{k+1}E_k
\]
for updating the correction part $F_{k+1}$ in $\varphi_{k+1}(A)$. The
above equation is easily implementable, moreover, representing $F_k$
in the form $F_k=Y_kW_k^T$, where $Y_k$ and $W_k$ are matrices with
infinitely many rows and a finite number of columns, and providing the
same representation for $E_k$ as $E_k=U_kV_k^T$, we may use the
updating equation
\begin{equation}\label{eq:yw}
Y_{k+1}=\left[Y_k \brb f_{k+1}U_k \right],\quad W_{k+1}=\left[ W_k \brb V_k \right].   
\end{equation}
Moreover, in order to keep low the number of columns in the matrices
$Y_{k+1}$ and $W_{k+1}$, one can apply a compression procedure based
on the rank-revealing QR factorization and on SVD, to the two matrices
in the right hand sides of \eqref{eq:yw}. This strategy has been
successfully used in \cite{bm:exp} in the case of the exponential
function.

Updating the Toeplitz part in $\varphi_k(A)$, i.e., computing
the coefficients of $\varphi_{k+1}(a(z))$ given those of 
$\varphi_k(a(z))$,
can be performed by means of the evaluation/in\-ter\-po\-lation technique
using as knots the roots of the unity of sufficiently large order.  In
fact, in this case we may rely on FFT to carry out the computation at
a low cost.

A similar computational strategy can be used if $f(x)$ is assigned as a
Laurent series in the form $\sum_{i\in\mathbb Z}f_ix^i$ so that $f(A)$
takes the form $f(A)=f_0I+\sum_{i=1}^\infty
(f_iA^i+f_{-i}A^{-i})$. Thus, once the matrix $A^{-1}$ has been
written in the form $A^{-1}=T(a^{-1})+E_{a^{-1}}$, one can apply the
above technique. Similar equations can be given in the case the
Toeplitz matrix is finite and has a sufficiently large size.

As an example to show the effectiveness of our approach, we performed
two numerical experiments. In the first one, we  applied the above
machinery to compute the exponential of the semi-infinite Toeplitz matrix 
$T(a)$ associated with the symbol $a(z)=\sum_{i=-1}^k z^i$ for
$k=1,2,\ldots,10$ corresponding to a Toeplitz matrix in Hessenberg form.
  In table \ref{tab:result} we report, besides the CPU
time in seconds, the values of the numerical
bandwidth of the exponential function, the dimension of the
non-negligible part of the correction $E_{\exp}$ and its rank.

We point out that the approximation of $\exp(T(a))$ represented in the
CQT form is quite good and that the CPU time needed for this
computation is particularly low. We observe also that the rank of the
correction has a moderate growth with respect to the band of $T(a)$.
\begin{table}
\centering
\begin{filecontents*}{semi-inf-exp.txt}
1	2.58e-02	37	18	18	7
2	2.58e-02	57	37	37	8
3	2.50e-02	86	32	65	9
4	2.61e-02	113	39	91	9
5	3.66e-02	142	46	121	9
6	3.18e-02	174	31	152	9
7	3.21e-02	236	45	210	12
8	3.47e-02	275	45	252	12
9	3.84e-02	316	45	288	12
10	4.15e-02	359	45	429	11
\end{filecontents*}
\pgfplotstabletypeset[          
columns={0,1,2,3,4,5}, 
columns/0/.style={
column name = $k$	
},
columns/1/.style={
column name = time,
       postproc cell content/.style={
       	      		/pgfplots/table/@cell content/.add={}{} 
       	      		}},   	         	      
columns/2/.style={column name =  band
}, 
columns/3/.style={
column name =  rows	
}, 
columns/4/.style={
column name =  columns	
},
columns/5/.style={
column name =  rank	
}
  ]{semi-inf-exp.txt}
         \caption{Computation of $\exp(T(a))$ where $a(z)=\sum_{i=-1}^k z^i$.}
         \label{tab:result}
       \end{table}

       In the second experiment, we consider matrices of finite size
       extending the CQT-arithmetic as pointed out in
       Section~\ref{sec:finite}. More precisely, we applied the power
       series definition for computing $\exp(A)$, where $A=H^{10}$ and $H$ is the
       $m\times m$ matrix $\hbox{trid}(1,2,1)/(2+2\cos(\frac{\pi}{m+1}))$.
       In the numerical test we have chosen increasing values of $m$
       as integer powers of $10$. Observe that, the matrix
       $A$ is diagonalizable by means of the sine
       transform. Therefore, for all the matrices in the algebra
       generated by A and for any function $f$, it is possible to
       retrieve a particular column of $f(A)$ with linear cost. In
       order to validate the results, we report ---as residual error--- the
       euclidean norm of the difference between the first column of
       the outcome and the first column of $\exp(A)$ computed by
       means of the sine transform.  Table \ref{tab:result1} shows the
       execution time in seconds, the residual errors, the Toeplitz bandwidth and the
       features of the correction.  Note that, the features of only
       one correction are reported because, due to the symmetry of
       $A$, the upper left and lower right corner corrections are
       equal.
\begin{table}
\centering
\begin{filecontents*}{exp.txt}
1e2	5.582200e-02	8.51e-16	87	81	49	15
1e3	5.447800e-02	2.08e-15	87	65	45	15
1e4	6.535800e-02	8.04e-16	87	65	45	15
1e5	8.222200e-02	1.87e-15	87	73	65	15
1e6	7.940500e-02	1.45e-15	87	46	65	15
1e7	8.184100e-02	1.04e-15	87	46	67	15
\end{filecontents*}
\pgfplotstabletypeset[          
columns={0,1,2,3,4,5,6}, 
columns/0/.style={
column name = Size	
},
columns/1/.style={
column name = time,
       postproc cell content/.style={
       	      		/pgfplots/table/@cell content/.add={}{}  
       	      		}},   	         	      
columns/2/.style={
column name = error 
},
columns/3/.style={column name =  band
}, 
columns/4/.style={
column name =  rows	
}, 
columns/5/.style={
column name =  columns	
},
columns/6/.style={
column name =  rank	
}
  ]{exp.txt}
         \caption{Computation of $\exp(A)$, with $A=H^{10}$ where $H=\hbox{trid}(1,2,1)/(2+2\cos(\frac{\pi}{m+1})) $ is an $m\times m$ tridiagonal matrix.}
         \label{tab:result1}
       \end{table}

\section{Function of a CQT matrix: the Dunford-Cauchy integral}\label{sec:3.2}

The definition of $f(A)$ based on the contour integral can be easily extended to infinite matrices which represent bounded operators \cite{shao2014finite,gil2008estimates}.

\begin{definition}
Let $A$ be a semi-infinite matrix which represents a bounded linear operator on $l^2(\mathbb Z^+)$ and let $\Lambda=\{z\in\mathbb C:\quad zI-A \hbox{ is not invertible}\}$ be its spectrum. Given an analytic function $f(x)$ defined on a compact domain $\Omega\supseteq \Lambda$ having boundary $\partial \Omega$,  $f(A)$ is defined as
\begin{equation}\label{def:int}
f(A):=\frac{1}{2\pi \cu}\int_{\partial \Omega}f(z)\mathfrak R(z)dz
\end{equation}
where $\mathfrak R(z):=(zI-A)^{-1}$ is the resolvent.
 \end{definition}

 The integral formula~\eqref{def:int} allows us to approximate $f(A)$ through a numerical integration scheme. That is, given a differentiable arc length parametrization $\gamma:[a,b]\rightarrow\mathbb C$  of $\partial \Omega$ we can write
\[
\frac{1}{2\pi \cu}\int_{\partial \Omega}f(z)\mathfrak R(z)dz=
 \int_{a}^b g(x) dx
\] 
 where $g(x):=\frac{1}{2\pi \cu}\gamma'(x)f(\gamma(x))\mathfrak R(\gamma(x))$ is a matrix valued function. The above integral can be approximated by means of a quadrature formula with nodes $x_k$ and weights $w_k$, i.e.,
 \begin{equation}\label{approxscheme}
 \int_{a}^b g(x) dx\ \approx \ \sum_{k = 1}^N w_k
          \cdot g(x_k).
 \end{equation}
 The approximation schemes are determined with the strategy of increasing the number of nodes until the required precision is reached. If the weights are non-negative, the approximation \eqref{approxscheme} converges for $N\to\infty$  to $f(A)$.

We consider the trapezoidal approximation scheme with a doubling strategy for the nodes. That is, we consider the double indexed family $\{x_k^{(n)},w_k^{(n)}\}$ such that:
\begin{itemize}
\item $n\in\mathbb Z^+$ and $k=1,\dots, 2^n+1$,
\item $a=x_1^{(n)}<x_2^{(n)}<\dots<x_{2^n+1}^{(n)}=b$ are equally spaced points in $[a,b]$ $\forall n\in \mathbb Z^+$,
\item $w_1^{(n)}=w_{2^n+1}^{(n)}=\frac{b-a}{2^{n+1}}$ and $w_{k}^{(n)}=\frac{b-a}{2^{n}}$, $k=2,\dots,2^n$.
\end{itemize} 
In particular, observe that the nodes at a certain step $n$ correspond
to those with odd indices at step $n+1$.
      
Using a trapezoidal approximation of the integral~\eqref{def:int} we
can prove that the function of a CQT-matrix is again a CQT-matrix.

        \begin{theorem}
          Let $A=T(a)+E_a$ be a {\rm CQT}-matrix with spectrum $\Lambda$ and symbol $a(z)\in\mathcal W_1$ . 
          Let $f(z)$ be an analytic function defined on the domain
          $\Omega\subset\mathbb C$ which encloses $\Lambda$ such that 
          $a(\mathbb T)\subset  \Omega$. Assume that
          $\partial \Omega$ admits a differentiable arc length
          parametrization $\gamma:[a,b]\rightarrow \partial \Omega$.
          Then $f(A)$ is a {\rm CQT}-matrix.
          
          Moreover, if $A\in\mathcal{AQT}$ then $f(A)\in\mathcal{AQT}$.
        \end{theorem}

        \begin{proof}
          Given the family $\{x_k,w_k\}$ of nodes and weights of the
          trapezoidal approximation scheme for \eqref{def:int} we
          consider the sequence of rational functions in $A$:
        \[
        \{r_n(A)\}_{n\in\mathbb N^+}=
        \left\lbrace
        \sum_{k=1}^{2^n+1}w_k^{(n)}g(x_k^{(n)})
        \right\rbrace_{n\in\mathbb N^+}= \left\lbrace
                \frac{b-a}{2^n}\sum_{k=1}^{2^n}g(x_k^{(n)})
                \right\rbrace_{n\in\mathbb N^+}
        \]
        where $g$ is defined according to \eqref{approxscheme} and the
        latter equality follows from the fact that $\partial \Omega$
        is a closed simple curve, thus $\gamma (x_1^{(n)})=\gamma
        (x_{2^{n+1}}^{(n)})$. This sequence is formed by CQT-matrices
        whose limit, if it exists, has a Toeplitz part with 
        symbol $f(a(z))$.  Therefore, it is sufficient to show that this
        sequence is Cauchy with respect to the norm $\|\cdot\|\aqt$.
        
        Consider the difference
        \[
        r_{n+1}(A)-r_n(A)=\frac{b-a}{2^{n+1}}\sum_{k=1}^{2^n}\left(g(x_{2k}^{(n+1)})-g(x_{2k-1}^{(n+1)})\right)
        \]
        and observe that (for notational simplicity we omit the superscript $(n+1)$ in the nodes)
        \[
        g(x_{2k})-g(x_{2k-1})=l(x_{2k})\mathfrak R(\gamma(x_{2k}))-l(x_{2k-1})\mathfrak R(\gamma(x_{2k-1}))
        \]
        where $l:[a,b]\rightarrow \mathbb C$, $l(x)=\frac1{2\pi\cu}\gamma'(x)f(\gamma(x))$. Assuming that $\gamma(x)$ has continuous second derivative, then $l(x)$  is a Lipschitz
        function. Indicating with $L$ the Lipschitz constant of $l$
        and defining $M:=\max_{\partial \Omega}\norm{\mathfrak
          R(z)}\aqt$, $G:=\max_{[a,b]}|l(x)|$ we get
        \begin{align*}
          \norm{ g(x_{2k})-&g(x_{2k-1})}\aqt \le\ |l(x_{2k})\mathfrak -l(x_{2k-1})|\cdot\norm{\mathfrak R(\gamma(x_{2k}))}\aqt \\
          &+|l(x_{2k-1})|\cdot \norm{\mathfrak R(\gamma(x_{2k}))-\mathfrak R(\gamma(x_{2k-1}))}\aqt \\
          \le&\ L|x_{2k}-x_{2k-1}|\cdot \norm{\mathfrak R(\gamma(x_{2k}))}\aqt \\
          &+|l(x_{2k-1})|\cdot|\gamma(x_{2k})-\gamma(x_{2k-1})|\cdot \norm{\mathfrak R(\gamma(x_{2k}))}\aqt \norm{\mathfrak R(\gamma(x_{2k-1}))}\aqt \\
          \le&\ \frac{LM(b-a)}{2^{n+1}}+ \frac{GM^2(b-a)}{2^{n+1}}
        \end{align*}
        where we used $|\gamma(x_{2k})-\gamma(x_{2k-1})|\leq |x_{2k}-x_{2k-1}|$ and the identity $\mathfrak R(z_1)-\mathfrak
        R(z_2)=(z_2-z_1)\mathfrak R(z_1)\mathfrak R(z_2)$.
        
        In particular, we can write
        \[
        \norm{r_{n+1}(A)-r_n(A)}\aqt
        \le\frac{b-a}{2^{n+1}}\sum_{k=1}^{2^n}\frac{(LM+GM^2)(b-a)}{2^{n+1}}=
        c\cdot 2^{-(n+2)}
        \]
        where $c:=(b-a)^2(LM+GM^2)$ is independent of $n$. Therefore, given $n_2>n_1$, we have
        \[\begin{split}
        \norm{r_{n_2}(A)-r_{n_1}(A)}\aqt &\le \sum_{j=n_1}^{n_2-1}\norm{r_{j+1}(A)-r_{j}(A)}\aqt \\
        &\le c\sum_{j=n_1}^{n_2-1}2^{-(j+2)}\leq c\cdot 2^{-(n_1+1)},
      \end{split}\] which proves that $ \{r_n(A)\}_{n\in\mathbb N^+}$
      is a Cauchy sequence in the Banach algebra of
      $\mathcal{CQT}$-matrices. By relying on the same arguments used in 
       the proof of Theorem \ref{th:7} we deduce that $g(z)=f(a(z))$. 
       So if $a(z)$ is analytic in a certain annulus $\mathbb A(r_a,R_a)$ containing $\mathbb T$ then there exists $\mathbb A(r,R)\subset\mathbb A(r_a,R_a) $ such that $a(\mathbb A(r,R))\subset \mathbb A(r_f,R_f)$. Thus the composed function $g(z)=f(a(z))$  is analytic in $\mathbb A(r,R)$. 
       This completes the proof.
        \end{proof}
        
\subsection{Computational aspects}
Numerical integration based on the trapezoidal rule at the roots of
unity can be easily implemented to approximate a matrix function
assigned in terms of a Dunford-Cauchy integral. In fact all the
operations involved in the computation reduce to performing matrix
additions, multiplication of a matrix by a scalar and matrix
inversion. The latter operation is the one with the highest
computational cost. 

We applied the contour integral definition for computing $\sqrt{I+H^{10}}$ and $\log(I+H^{10})$ where $H$ is the $m\times m$ matrix $H=\hbox{trid}(1,2,1)/(2+2\cos(\frac{\pi}{m+1}))$ considered in Section~\ref{sec:exp1}. We used the trapezoidal rule with a doubling strategy for the nodes for integrating on a disc which contains the spectrum of $I+H^{10}$. Since $H$ is rescaled to have spectrum in $[0,1]$, we selected as center of the disc $1.5$ and radius $1$. Table \ref{tab:result2}-\ref{tab:result3} report the execution time, the residuals, the Toeplitz bandwidth and the features of the correction as the size of the argument increases exponentially. Once again, we reported only the features of one correction because, due to the symmetry of $A$, the upper left and lower right corner corrections are equal.

\begin{table}
\centering
\begin{filecontents*}{sqrt.txt}
1e+02	5.826100e-02	8.51e-16	155	89	90	15
1e+03	7.932800e-02	2.08e-15	79	89	90	15	
1e+04	8.148800e-02	8.04e-16	79	89	90	15
1e+05	6.871300e-02	1.87e-15	79	89	89	15
1e+06	8.368000e-02	1.45e-15	79	89	89	15
1e+07	8.270500e-02	1.04e-15	79	89	90	15
\end{filecontents*}
\pgfplotstabletypeset[          
columns={0,1,2,3,4,5,6}, 
columns/0/.style={
column name = Size	
},
columns/1/.style={
column name = time,
       postproc cell content/.style={
       	      		/pgfplots/table/@cell content/.add={}{}
       	      		}},   	         	      
columns/2/.style={
column name = error 
},
columns/3/.style={column name =  band
}, 
columns/4/.style={
column name =  rows	
}, 
columns/5/.style={
column name =  columns	
},
columns/6/.style={
column name =  rank	
}
  ]{sqrt.txt}
         \caption{Computation of $\sqrt{A}$, with $A=I+H^{10}$ where  $H=\hbox{trid}(1,2,1)/(2+2\cos(\frac{\pi}{m+1})) $ is an $m\times m$ tridiagonal matrix..}
         \label{tab:result2}
       \end{table}
       \begin{table}
       \centering
       \begin{filecontents*}{log.txt}
1e+02	1.900574e+00	5.571234e-14	87	89	90	15  
1e+03	1.877712e+00	5.501482e-14	159	90	90	15  
1e+04	1.529042e+00	5.570115e-14	159	89	90	15
1e+05	1.618496e+00	5.561008e-14	159	89	89	15
1e+06	1.992006e+00	5.536176e-14	159	90	89	15
1e+07	1.649571e+00	5.559155e-14	159	89	90	15 
\end{filecontents*}
\pgfplotstabletypeset[          
       columns={0,1,2,3,4,5,6}, 
       columns/0/.style={
       column name = Size	
       },
       columns/1/.style={
       column name = time,
              postproc cell content/.style={
              	      		/pgfplots/table/@cell content/.add={}{}
              	      		}
  		 },   	         	      
       columns/2/.style={
       column name = error 
       },
       columns/3/.style={column name =  band
       }, 
       columns/4/.style={
       column name =  rows	
       }, 
       columns/5/.style={
       column name =  columns	
       },
       columns/6/.style={
       column name =  rank	
       }
         ]{log.txt}
                \caption{Computation of $\log(A)$,  with $A=I+H^{10}$ where  $H=\hbox{trid}(1,2,1)/(2+2\cos(\frac{\pi}{m+1})) $ is an $m\times m$ tridiagonal matrix.}
                \label{tab:result3}
              \end{table}
              
\section{Conclusions}\label{sec:conc}
We have extended the concept of matrix function to CQT matrices, i.e.,
infinite matrices of the form $A=T(a)+E$, by showing that, 
under suitable mild assumptions, for a CQT
matrix $A$ and for a function $f(x)$ expressed either in terms of a
power (Laurent) series, or in terms of the Dunford-Cauchy integral,
the matrix function $f(A)$ is still a CQT matrix. We have outlined
algorithms for the computation of $f(A)$. This
approach has been adapted to the case of $f(A_m)$ where $A_m$ is the $m\times m$ leading
principal submatrix of $A$.

Among the open issues that will be part of our research interests, it
would be interesting to analyze the behavior of the singular values
$\sigma_i^{(k)}$ of the $m\times m$ truncation of the correction $E_k$
such that $(T(a)+E)^k=T(a^k)+E_k$, and relate the decay of these
values for $i=1,2,\ldots, m$ and for $k=1,2,\ldots$, to the
qualitative properties of the function $a(z)$.  
In fact, from the numerical experiments that we have
performed with several functions $a(z)$, it turns out that the
numerical rank of $E_k$ remains bounded by a constant independent of
$k$.

\end{document}